\documentclass[british,a4paper]{scrartcl}
\usepackage[T1]{fontenc}
\usepackage[latin9]{inputenc}
\usepackage[a4paper]{geometry}
\usepackage{babel}
\usepackage{prettyref}
\usepackage{mathtools}
\usepackage{amsmath}
\usepackage{amsthm}
\usepackage{amssymb}
\usepackage[unicode=true,pdfusetitle,
 bookmarks=true,bookmarksnumbered=false,bookmarksopen=false,
 breaklinks=false,pdfborder={0 0 1},backref=false,colorlinks=false]
 {hyperref}

\makeatletter
\theoremstyle{definition}
\newtheorem*{defn*}{\protect\definitionname}
\theoremstyle{plain}
\newtheorem{thm}{\protect\theoremname}[section]
\theoremstyle{plain}
\newtheorem{prop}[thm]{\protect\propositionname}
\theoremstyle{plain}
\newtheorem{cor}[thm]{\protect\corollaryname}
\theoremstyle{remark}
\newtheorem{rem}[thm]{\protect\remarkname}
\theoremstyle{definition}
\newtheorem{example}[thm]{\protect\examplename}

\@ifundefined{date}{}{\date{}}
\usepackage{prettyref}

\usepackage{enumerate}
\allowdisplaybreaks

\newcommand{\R}{\mathbb{R}}
\newcommand{\N}{\mathbb{N}}

\newcommand{\dom}{\operatorname{dom}}
\newcommand{\ran}{\operatorname{ran}}

\renewcommand{\d}{\,\mathrm{d}}

\newcommand{\grad}{\operatorname{grad}}

\newcommand{\curl}{\operatorname{curl}}
\newcommand{\dive}{\operatorname{div}}
\renewcommand{\tilde}{\widetilde}

\newcommand{\Span}{\operatorname{span}}

\newcommand{\red}{\mathrm{red}}
\DeclareMathAccent{\Circ}{\mathalpha}{operators}{"17}

\newrefformat{subsec}{Subsection \ref{#1}}
\newrefformat{prob}{Problem \ref{#1}}
\newrefformat{prop}{Proposition \ref{#1}}
\newrefformat{lem}{Lemma \ref{#1}}
\newrefformat{thm}{Theorem \ref{#1}}
\newrefformat{cor}{Corollary \ref{#1}}
\newrefformat{rem}{Remark \ref{#1}}
\newrefformat{exa}{Example \ref{#1}}
\newrefformat{sub}{Subsection \ref{#1}}
\newrefformat{eq}{(\ref{#1})}

\theoremstyle{definition}

\newrefformat{hyp}{Hypotheses \ref{#1}}

\makeatother

\providecommand{\corollaryname}{Corollary}
\providecommand{\definitionname}{Definition}
\providecommand{\examplename}{Example}
\providecommand{\propositionname}{Proposition}
\providecommand{\remarkname}{Remark}
\providecommand{\theoremname}{Theorem}

\begin{document}
\title{A Note on Some Non-Local Boundary Conditions and their Use in Connection
with Beltrami Fields }
\author{Rainer Picard\thanks{Institut für Analysis, TU Dresden, Dresden, Germany, rainer.picard@tu-dresden.de}
and Sascha Trostorff\thanks{Mathematisches Seminar, CAU Kiel, Kiel, Germany, trostorff@math.uni-kiel.de}}
\maketitle
\begin{abstract}
\textbf{Abstract. }We consider two operators $A_{0},B_{0}$ between
two Hilbert spaces satisfying $A_{0}\subseteq-B_{0}^{\ast}$ and $B_{0}\subseteq-A_{0}^{\ast}$
and inspect extensions $A^{\#}$ and $B^{\#}$ of $A_{0}$ and $B_{0}$,
respectively, whose domain consists of those elements satisfying an
abstract periodic boundary condition. The motivating example is the
derivative on some interval, where the so-defined realisation gives
the classical derivative with periodic boundary conditions. We derive
necessary and sufficient conditions for the operator equality $A^{\#}=-\left(B^{\#}\right)^{\ast}$
and illustrate our findings by applications to the classical vector
analytic operators $\grad,\dive$ and $\curl$. In particular, the
realisation $\curl^{\#}$ naturally arises in the study of so-called
Beltrami fields. 
\end{abstract}
\tableofcontents{}

\section{Introduction}

A typical non-local boundary condition we have in mind is the condition
of periodicity -- say -- on the unit interval $\left]-1/2,1/2\right[$
for the standard 1-dimensional derivative $\partial$. In an $L_{2}\left(\left]-1/2,1/2\right[\right)$-setting
this condition can conveniently be described (see Example \ref{exa:periodicity-boundary-conditions})
by
\begin{equation}
\partial u\perp1,\label{eq:period0}
\end{equation}
resulting in a skew-selfadjoint operator $\partial^{\#}$. If we denote
by $\partial_{0}$ the derivative with vanishing boundary data, we
have 
\[
\partial=-\partial_{0}^{*}.
\]
We clearly have the orthogonal decomposition
\begin{align*}
L_{2}\left(\left]-1/2,1/2\right[\,,\mathbb{R}\right) & =\overline{\ran}\left(\partial_{0}\right)\oplus\ker\left(\partial\right),\\
 & =\overline{\ran}\left(\partial_{0}\right)\oplus\mathbb{R},
\end{align*}
so that the non-local boundary condition \eqref{eq:period0} defining
$\partial^{\#}$ can be rephrased as
\begin{equation}
\partial u\in\overline{\ran}\left(\partial_{0}\right).\label{eq:period1}
\end{equation}
Remarkably this is essentially the same situation as in the case of
the vector-analytical operator $\curl$ with boundary conditions associated
with the topic of force-free magnetic field, the eigensolutions of
which are frequently referred to as Beltrami fields. In classical
terms this boundary condition is $n\cdot\curl H=0$ on the boundary
of a domain $\Omega$, $n$ denoting the unit normal vector field.
In an $L_{2}\left(\Omega;\mathbb{R}^{3}\right)$-setting this boundary
condition can be formulated as an orthogonality condition 
\[
\curl H\perp\ran\left(\grad\right).
\]
In topologically more complex domains one needs to require additional
constraints to obtain a reasonably small spectrum, which leads to
the analogue to \eqref{eq:period1}
\[
\curl H\in\overline{\ran}\left(\curl_{0}\right).
\]
The resulting operator $\curl^{\#}$ has been extensively studied,
see the discussion in \cite{Picard1998,Picard1998_exterior}. After
a long pre-history based on assumptions warranting that $\ran\left(\curl_{0}\right)$
is actually closed, N. Filonov, \cite{Filonov1999}, has found, based
on potential theoretical considerations, that it suffices to assume
that $\Omega$ is merely of finite measure. This is in sharp contrast
to the situation of $\ran\left(\curl_{0}\right)$ closed, for which
at least some boundary regularity appears to be required.

Filonov's result suggests that there actually might be an abstract
functional analytical mechanism in the background. In this paper we
shall indeed present such an approach, which allows to cover the case
of bounded domains. Exterior domains clearly require adjustments and
it appears that to control the asymptotics of unbounded domains of
finite measure more concrete methods, such as potential theory, are
required.

In the following we shall develop an abstract setting covering these
examples and expanding the reach of the concept to a larger class
of applications. In Section \ref{sec:Preliminaries} we first collect
some basic facts useful for our framework. With the periodic boundary
condition case as a root example in mind, we shall speak of `abstract
periodicity', which we introduce in Section \ref{sec:Abstract-Periodicity}.
In Section \ref{sec:Conditions-for-a} we deal with the closed range
property of $A^{\#}$, which is a crucial ingredient of linear solution
theory. The last section, Section \ref{sec:Applications}, is dedicated
to some illustrative examples.

\section{Preliminaries\label{sec:Preliminaries}}

In this section we collect some probably well-known results, which
will be useful in further sections of this note (see e.g. \cite[Lemmas 4.1, 4.3]{Pauly2019}).
For the readers convenience we include the proofs. Throughout, let
$H_{0},H_{1}$ be Hilbert spaces and $A\colon\dom(A)\subseteq H_{0}\to H_{1}$
a closed, densely defined, linear operator. We begin to introduce
the reduced operator for $A$.
\begin{defn*}
We define the \emph{reduced operator} $A_{\red}$ by 
\[
A_{\red}\colon\dom(A)\cap\ker(A)^{\bot}\subseteq\ker(A)^{\bot}\to\overline{\ran}(A),\quad x\mapsto Ax.
\]

Moreover, we set $D_{A_{\red}}\coloneqq\dom(A)\cap\ker(A)^{\bot}$
and equip it with the graph norm of $A_{\red}$. 
\end{defn*}
Clearly, $A_{\red}$ is one-to-one and has dense range. Moreover,
$A_{\red}$ is closed and hence, $D_{A_{\red}}$ is complete.
\begin{prop}
\label{prop:reduced_op}We consider $A$ and the reduced operator
$A_{\red}.$ Then

\begin{enumerate}[(a)]

\item $\ran(A)$ is closed if and only if $A_{\red}$ is boundedly
invertible.

\item $D_{A_{\red}}\hookrightarrow\hookrightarrow H_{0}$ if and
only if $A_{\red}$ is compactly invertible. 

\end{enumerate}
\end{prop}

\begin{proof}
(a): If $\ran(A)$ is closed, then $A_{\red}$ is onto and hence,
bijective. Moreover, since $A_{\red}$ is closed, the bounded invertibility
is a consequence of the closed graph theorem. If, on the other hand,
$A_{\red}$ is boundedly invertible, then it is onto, which yields
$\overline{\ran}(A)=\ran(A_{\red})=\ran(A),$ hence $A$ has closed
range.\\
(b): Assume that $D_{A_{\red}}\hookrightarrow\hookrightarrow H_{0}$.
Since $D_{A_{\red}}\subseteq\ker(A)^{\bot}$ and $\ker(A)^{\bot}$
is closed in $H_{0}$, we infer that $D_{A_{\red}}\hookrightarrow\hookrightarrow\ker(A)^{\bot}.$
We claim that $A_{\red}$ is boundedly invertible. Indeed, if $A_{\red}$
is not boundedly invertible, we find a sequence $(x_{n})_{n\in\N}$
in $\dom(A_{\red})$ such that $A_{\red}x_{n}\to0$ and $\|x_{n}\|_{H_{0}}=1$
for each $n\in\N.$ The latter gives that $(x_{n})_{n}$ is a bounded
sequence in $D_{A_{\red}}$ and thus, it possesses a convergent sub-sequence.
So, we assume without loss of generality that $x_{n}\to x$ for some
$x\in\ker(A)^{\bot}.$ By the closedness of $A_{\red}$, we obtain
$x\in\dom(A_{\red})$ and $A_{\red}x=0.$ Thus, $x=0,$ which contradicts
$\|x\|=\lim_{n\to\infty}\|x_{n}\|=1.$ Thus, $A_{\red}$ is boundedly
invertible. Finally, $A_{\red}^{-1}:\ran(A)\to D_{A_{\red}}$ is bounded
and hence, $A_{\red}^{-1}\colon\ran(A)\to\ker(A)^{\bot}$ is compact.
\\
Conversely, assume that $A_{\red}$ is compactly invertible and set
$\iota:D_{A_{\red}}\to\ker(A)^{\bot}$ by $\iota x=x$. Consider $A_{\red}\colon D_{A_{\red}}\to\ran(A)$,
which is a bounded operator and hence, $\iota=A_{\red}^{-1}A_{\red}\colon D_{A_{\red}}\to\ker(A)^{\bot}$
is compact, which in turn is equivalent to $D_{A_{\red}}\hookrightarrow\hookrightarrow H_{0}.$ 
\end{proof}
The following proposition can also be found in \cite[Lemma 2.4]{Trostorff2014}.
\begin{prop}
\label{prop:adjoint_reduced}We have 
\[
(A_{\red})^{\ast}=\left(A^{\ast}\right)_{\red}.
\]
\end{prop}

\begin{proof}
Let $y\in\dom\left(A_{\red}^{\ast}\right).$ Then for each $x\in\dom(A_{\red})=\dom(A)\cap\ker(A)^{\bot}$
we have 
\[
\langle A_{\red}x,y\rangle=\langle x,\left(A_{\red}\right)^{\ast}y\rangle.
\]
Denote by $P:H_{0}\to H_{0}$ the projection onto $\ker(A)^{\bot}=\overline{\ran}(A^{\ast}).$
Then we obtain for each $x\in\dom(A)$
\[
\langle Ax,y\rangle=\langle A_{\red}Px,y\rangle=\langle Px,(A_{\red})^{\ast}y\rangle=\langle x,(A_{\red})^{\ast}y\rangle,
\]
where we have used $(A_{\red})^{\ast}y\in\ker(A)^{\bot}$ in the last
equality. Hence, $y\in\dom(A^{\ast})$ with $A^{\ast}y=(A_{\red})^{\ast}y.$
Since $y\in\overline{\ran}(A)=\ker(A^{\ast})^{\bot}$ by definition,
we infer that $y\in\dom\left(\left(A^{\ast}\right)_{\red}\right)$
and thus, $\left(A^{\ast}\right)_{\red}y=A^{\ast}y=(A_{\red})^{\ast}y,$
which shows $(A_{\red})^{\ast}\subseteq(A^{\ast})_{\red}.$ For the
reverse inclusion, let $y\in\dom((A^{\ast})_{\red});$ i.e. $y\in\dom(A^{\ast})\cap\ker(A^{\ast})^{\bot}.$
Then for each $x\in\dom(A_{\red})=\dom(A)\cap\ker(A)^{\bot}$ we compute
\[
\langle A_{\red}x,y\rangle=\langle Ax,y\rangle=\langle x,A^{\ast}y\rangle,
\]
and since $A^{\ast}y\in\ran(A^{\ast})\subseteq\ker(A)^{\bot},$ we
infer $y\in\dom((A_{\red})^{\ast}),$ which shows the assertion.
\end{proof}
\begin{cor}
\label{cor:closed_adjoint}\begin{enumerate}[(a)]

\item $\ran(A)$ is closed if and only if $\ran(A^{\ast})$ is closed.

\item $D_{A_{\red}}\hookrightarrow\hookrightarrow H_{0}$ if and
only if $D_{A_{\red}^{\ast}}\hookrightarrow\hookrightarrow H_{1}.$

\end{enumerate}
\end{cor}

\begin{proof}
This is a direct consequence of \prettyref{prop:reduced_op} and \prettyref{prop:adjoint_reduced}
and the fact, that the adjoint of a compact operator is again compact.
\end{proof}
We now focus on the case when $A$ is selfadjoint. Note that then
$H_{0}=H_{1}$ and $\overline{\ran}(A)=\overline{\ran}(A^{\ast})=\ker(A)^{\bot}.$
Hence, $A$ and $A_{\red}$ are operators operators acting on one
Hilbert space.
\begin{prop}
\label{prop:spectrum_reduced}Let $A$ be selfadjoint. Then 
\[
\sigma(A)\cup\{0\}=\sigma(A_{\red})\cup\{0\}.
\]
\end{prop}

\begin{proof}
Let $\lambda\in\rho(A_{\red})$ with $\lambda\ne0.$ We show that
$\lambda\in\rho(A)$; that is, we show the bijectivity of $\lambda-A.$
First, $\lambda-A$ is one-to-one, since $\lambda x=Ax$ for $x\in\dom(A)$
implies $x\in\ker(A),$ since otherwise $\lambda x=A_{\red}x$. However,
if $x\in\ker(A)$ then $\lambda x=Ax=0$ and since $\lambda\ne0,$
we infer $x=0.$ Let now $y\in H_{1}$ and denote by $P\colon H_{0}\to H_{0}$
the projector onto $\ker(A)^{\bot}=\overline{\ran}(A).$ We set $x_{0}\coloneqq(\lambda-A_{\red})^{-1}Pf$
and $x_{1}=\frac{1}{\lambda}(1-P)f.$ Then $x_{0}\in\dom(A_{\red})=\dom(A)\cap\ker(A)^{\bot}$and
$x_{1}\in\ker(A)\subseteq\dom(A)$. Setting $x\coloneqq x_{0}+x_{1}\in\dom(A)$
we obtain 
\[
(\lambda-A)x=(\lambda-A)x_{0}+\lambda x_{1}=Pf+(1-P)f=f,
\]
and hence $\lambda-A$ is onto. \\
Assume now conversely that $\lambda\in\rho(A)$ with $\lambda\ne0.$
Then clearly, $\lambda-A_{\red}$ is one-to-one and for $f\in\overline{\ran}(A)$
we set 
\[
x\coloneqq(\lambda-A)^{-1}f.
\]
Then 
\[
\lambda x=Ax+f\in\overline{\ran}(A)=\ker(A)^{\bot}
\]
 and since $\lambda\ne0,$ we infer $x\in\ker(A)^{\bot}.$ Hence,
$x\in\dom(A_{\red})$ and clearly, $(\lambda-A_{\red})x=f,$ which
proves $\lambda\in\rho(A_{\red}).$ 
\end{proof}
\begin{cor}
\label{cor:isolated_value}Assume that $A$ is selfadjoint. Then $\ran(A)$
is closed if and only if there exists $r>0$ such that $B(0,r)\cap\sigma(A)\subseteq\{0\}.$
\end{cor}

\begin{proof}
First note that $A_{\red}$ is selfadjoint and one-to-one. Hence,
$0$ is not an isolated value of $\sigma(A_{\red})$ (see e.g. \cite[Corollary 5.11]{Schmuedgen2012}).
Thus, by \prettyref{prop:spectrum_reduced} $B(0,r)\cap\sigma(A)\subseteq\{0\}$
is equivalent to $0\in\rho(A_{\red}),$ which by \prettyref{prop:reduced_op}
is equivalent to the closedness of $\ran(A).$ 
\end{proof}

\section{Abstract Periodicity\label{sec:Abstract-Periodicity}}

Let $H_{0},H_{1}$ be two Hilbert spaces and $A_{c}\colon\dom(A_{c})\subseteq H_{0}\to H_{1}$
as well as $B_{c}\colon\dom(B_{c})\subseteq H_{1}\to H_{0}$ be two
densely defined linear operators satisfying 
\[
A_{c}\subseteq-B_{c}^{\ast}.
\]

By taking adjoints, we obtain 
\[
B_{c}\subseteq-A_{c}^{\ast}
\]
and we set 
\[
A\coloneqq-B_{c}^{\ast}\quad B\coloneqq-A_{c}^{\ast}.
\]
 As a consequence, $A_{c}\subseteq A$ and $B_{c}\subseteq B$ are
closable and we set $A_{0}\coloneqq\overline{A_{c}}$ and $B_{0}\coloneqq\overline{B_{c}}$.
Note that $A=-B_{0}^{\ast}$ and $B=-A_{0}^{\ast}$.
\begin{rem}
\label{rem:grad_div}The abstract setting above reflects the classical
definition of linear differential operators with and without boundary
conditions. Indeed, if $H_{0}=L_{2}(\Omega),H_{1}=L_{2}(\Omega)^{n}$
for some open $\Omega\subseteq\R^{n}$ we can set $A_{c}\colon C_{c}^{\infty}(\Omega)\subseteq L_{2}(\Omega)\to L_{2}(\Omega)^{n}$
by $A_{c}\phi\coloneqq\grad\phi$, where $C_{c}^{\infty}(\Omega)$
denotes the space of compactly supported infinitely often differentiable
functions on $\Omega$ and $B_{c}\colon C_{c}^{\infty}(\Omega)^{n}\subseteq L_{2}(\Omega)^{n}\to L_{2}(\Omega)$
by $B_{c}\Psi\coloneqq\dive\Psi$. Then integration by parts yields
$A_{c}\subseteq-B_{c}^{\ast}$ and by definition $A=-B_{c}^{\ast}=\grad$
with domain $H^{1}(\Omega).$ Likewise, $A_{0}=\overline{A_{c}}$
is the gradient with domain $H_{0}^{1}(\Omega)$; that is, the elements
satisfy an abstract homogeneous Dirichlet boundary condition. Similarly
$B=-A_{c}^{\ast}=\dive$ with maximal domain; i.e., 
\[
\dom(B)=\{\Psi\in L_{2}(\Omega)^{n}\,;\,\dive\Psi\in L_{2}(\Omega)\},
\]
where the divergence is defined in the distributional sense and the
elements in $\dom(B_{0})$ satisfy a generalised homogeneous Neumann
boundary condition (see \cite{Picard_McGhee,STW2022} for details).
\end{rem}

The focus of this section is on the following restrictions of $A$
and $B$.
\begin{defn*}
We define the restrictions $A^{\#}$ and $B^{\#}$ of $A$ and $B$,
respectively, by the domains 
\begin{align*}
\dom(A^{\#}) & \coloneqq\{x\in\dom(A)\,;\,Ax\in\overline{\ran}(A_{0})\}\\
\dom(B^{\#}) & \coloneqq\{y\in\dom(B)\,;\,Bx\in\overline{\ran}(B_{0})\}.
\end{align*}
\end{defn*}
\begin{rem}
Note that $\overline{\ran}(A_{0})=\left(\ker A_{0}^{\ast}\right)^{\bot}=\left(\ker B\right)^{\bot}$
and thus, 
\[
\dom(A^{\#})=\{x\in\dom(A)\,;\,Ax\bot\ker B\},
\]
and likewise for $B^{\#}.$ Moreover, by definition $A_{0}\subseteq A^{\#}\subseteq A$
and $B_{0}\subseteq B^{\#}\subseteq B.$ Finally, it is immediate
that $A^{\#}$ and $B^{\#}$ are closed.
\end{rem}

\begin{example}
\label{exa:periodicity-boundary-conditions}The boundary conditions
induced by the domain constraints of $A^{\#}$ can be interpreted
as an abstract version of periodicity. Indeed, if $\Omega=]-1/2,1/2[$,
then as discussed already in the introduction, $A_{0}=B_{0}=\partial_{0}$
and $A=B=\partial,$ where $\dom(A)=H^{1}(]-1/2,1/2[)$ and $\dom(A_{0})=H_{0}^{1}(]-1/2,1/2[)=\{u\in H^{1}(]-1/2,1/2[)\,;\,u(-1/2)=u(1/2)=0\}$.
Then $u\in\dom(A^{\#})$ if and only if $u\in H^{1}(]-1/2,1/2[)$
and 
\[
\partial u\bot\ker\partial.
\]
Since $\ker\partial=\Span\{1\}$, we infer that 
\[
\int_{-1/2}^{1/2}\partial u(t)\d t=0,
\]
which is equivalent to $u(1/2)=u(-1/2).$ Hence 
\[
\dom(A^{\#})=\{u\in H^{1}(]-1/2,1/2[)\,;\,u(-1/2)=u(1/2)\}.
\]
\end{example}

It is our aim of this section to discuss, when $A^{\#}=-\left(B^{\#}\right)^{\ast}$.
First note that this cannot hold in general as the next example illustrates.
\begin{example}
We again consider the derivatives $\partial_{0}$ and $\partial$
but this time on the interval $[0,\infty[.$ More precisely, we set
$A_{0}=B_{0}\coloneqq\partial_{0}$ with domain $H_{0}^{1}([0,\infty[)=\{u\in H^{1}([0,\infty[)\,;\,u(0)=0\}$
and $A=B=\partial$ with domain $H^{1}([0,\infty[).$ Then $\ker\partial=\{0\}$
and thus, $u\in\dom(\partial^{\#})$ is equivalent to $u\in\dom(\partial)$
and thus, $\partial^{\#}=\partial.$ Since $\partial^{\ast}=-\partial_{0}\ne-\partial,$
the desired relation for the operator $\partial^{\#}$ cannot hold.
\end{example}

In order to characterise when $A^{\#}=-(B^{\#})^{\ast}$ actually
holds, we start with the following observation.
\begin{prop}
\label{prop:adjoint_B=000023}We have that 
\[
\overline{A|_{\ker(A)+\dom(A_{0})}}=-\left(B^{\#}\right)^{\ast}.
\]
\end{prop}

\begin{proof}
For notational convenience, we set $\tilde{A}\coloneqq A|_{\ker(A)+\dom(A_{0})}.$
Since $\tilde{A}\subseteq A,$ we infer that $\tilde{A}$ is closable.
Thus, the statement we want to show is equivalent to 
\[
-B^{\#}=\tilde{A}^{\ast}.
\]
Let $y\in\dom(\tilde{A}^{\ast})$. Then, for all $x\in\dom(\tilde{A})=\dom(A_{0})+\ker A$
we have 
\[
\langle\tilde{A}x,y\rangle=\langle x,\tilde{A}^{\ast}y\rangle.
\]
Choosing $x\in\dom(A_{0}),$ we infer that 
\[
\langle A_{0}x,y\rangle=\langle x,\tilde{A}^{\ast}y\rangle
\]
and hence, $y\in\dom(A_{0}^{\ast})=\dom(B)$ and $\tilde{A}^{\ast}y=-By.$
Moreover, choosing $x\in\ker A,$ we infer 
\[
\langle x,\tilde{A}^{\ast}y\rangle=0,
\]
that is, $-By=\tilde{A}^{\ast}y\in\left(\ker A\right)^{\bot}=\overline{\ran}(B_{0}),$
which yields $y\in\dom(B^{\#})$ and hence, $\tilde{A}^{\ast}\subseteq-B^{\#}.$
For the other inclusion, let $y\in\dom(B^{\#}).$ Then for $x\in\dom(\tilde{A})$,
that is, $x=x_{0}+x_{1}$ with $x_{0}\in\dom(A_{0})$ and $x_{1}\in\ker(A)$,
we compute 
\[
\langle y,\tilde{A}x\rangle=\langle y,A_{0}x_{0}\rangle=\langle-B^{\#}y,x_{0}\rangle=\langle-B^{\#}y,x\rangle,
\]
where we have used $B^{\#}y\in\left(\ker A\right)^{\bot}$ in the
last equality. This shows $-B^{\#}\subseteq\tilde{A}^{\ast}.$ 
\end{proof}
The latter result shows that always $-(B^{\#})^{\ast}\subseteq A^{\#}$,
since $A|_{\ker(A)+\dom(A_{0})}\subseteq A^{\#}$ and equality holds,
if and only if $\ker(A)+\dom(A_{0})$ is a core for $A^{\#}$. We
inspect this property a bit closer.
\begin{prop}
The set $\ker(A)+\dom(A_{0})$ is a core for $A^{\#}$ if and only
if $1\notin P\sigma(B^{\#}A^{\#})$, i.e., $1$ in no eigenvalue of
$B^{\#}A^{\#}.$ 
\end{prop}

\begin{proof}
We remark that $\ker(A)+\dom(A_{0})$ is a core for $A^{\#}$ if and
only if $\ker(A)+\dom(A_{0})$ is dense in $\dom(A^{\#})$ with respect
to the graph norm of $A^{\#}$, which in turn is equivalent to 
\[
\left(\ker(A)+\dom(A_{0})\right)^{\bot_{A^{\#}}}=\{0\},
\]
where the orthogonal complement is taken in the graph inner product
of $A^{\#}.$ We characterise the set on the left-hand side as follows:
\begin{align*}
x\in\left(\ker(A)+\dom(A_{0})\right)^{\bot_{A^{\#}}} & \Leftrightarrow\forall y\in\ker(A)+\dom(A_{0}):\,\langle x,y\rangle+\langle A^{\#}x,A^{\#}y\rangle=0,\\
 & \Leftrightarrow x\in\ker(A)^{\bot}\wedge\forall y\in\dom(A_{0}):\,\langle x,y\rangle+\langle A^{\#}x,A_{0}y\rangle=0,\\
 & \Leftrightarrow x\in\overline{\ran}(B_{0})\wedge A^{\#}x\in\dom(B)\wedge BA^{\#}x=x,\\
 & \Leftrightarrow A^{\#}x\in\dom(B^{\#})\wedge B^{\#}A^{\#}x=x,\\
 & \Leftrightarrow x\in\ker(1-B^{\#}A^{\#}).
\end{align*}
Thus, $\ker(A)+\dom(A_{0})$ is a core for $A^{\#}$ if and only if
$\ker(1-B^{\#}A^{\#})=\{0\},$ which means $1\notin P\sigma(B^{\#}A^{\#}).$
\end{proof}
Next, we provide a sufficient condition for $\ker(A)+\dom(A_{0})$
being a core for $A^{\#}$.
\begin{prop}
\label{prop:core_for_A=000023}Assume that $\ran(A^{\#})$ is closed.
Then the space $\dom(A_{0})+\ker A\subseteq\dom(A^{\#})$ is a core
for $A^{\#}$.
\end{prop}

To prove the latter proposition, we show a more general theorem, which
can by applied in the above situation.
\begin{thm}
\label{thm:weak=00003Dstrong}Let $S\colon\dom(S)\subseteq H_{0}\to H_{1}$
and $T\colon\dom(T)\subseteq H_{0}\to H_{1}$ be two densely defined
linear operators with $S\subseteq T.$ Assume further that 
\begin{itemize}
\item $T$ is closed with closed range,
\item $\ran(S)$ lies dense in $\ran(T),$
\item $\ker(S)=\ker(T)$.
\end{itemize}
Then $T=\overline{S}.$ 
\end{thm}

\begin{proof}
The inclusion $\overline{S}\subseteq T$ is obvious. For the other
inclusion, we recall that $T_{\red}$ is boundedly invertible by \prettyref{prop:reduced_op}.
Let now $x\in\dom(T)$ and decompose $x=x_{0}+x_{1}$ with $x_{0}\in\ker(T)$
and $x_{1}\in\ker(T)^{\bot}.$ Since $x_{0}\in\ker(S)\subseteq\dom(S)$
it suffices to show $x_{1}\in\dom(\overline{S}).$ For this, we choose
a sequence $(u_{n})_{n\in\N}$ in $\dom(S)$ with $Su_{n}\to Tx_{1}.$
Let $P$ denote the projection onto $\ker(S)^{\bot}=\ker(T)^{\bot}.$
Then 
\[
T_{\red}Pu_{n}=Tu_{n}=Su_{n}\to Tx_{1}=T_{\red}x_{1}
\]
 and hence, by the bounded invertibility of $T_{\red},$ we infer
$Pu_{n}\to x_{1}$. Since $Pu_{n}\in\ker(S)^{\bot}$ we infer, that
$Pu_{n}\in\dom(S)$ and hence $x_{1}\in\dom(\overline{S}),$ since
$SPu_{n}=Su_{n}\to Tx_{1}.$ 
\end{proof}
\begin{proof}[Proof of \prettyref{prop:core_for_A=000023}]
 We apply \prettyref{thm:weak=00003Dstrong} to $S=A|_{\ker(A)+\dom(A_{0})}$
and $T=A^{\#}.$ Then clearly, $S\subseteq T$ and $T$ has a closed
range by assumption. Moreover, $\ran(S)=\ran(A_{0})$ and since $\ran(T)=\ran(A^{\#})\subseteq\overline{\ran(A_{0})},$
the range of $S$ lies dense in the range of $T$. Finally, we have
that $A^{\#}x=0$ for some $x\in\dom(A^{\#})$ implies $x\in\ker(A)$
and thus,~$x\in\dom(S)$ with $Sx=0.$ Thus, $\ker S=\ker T$ and
\prettyref{thm:weak=00003Dstrong} yields the assertion. 
\end{proof}
We summarise the result of this section in the next theorem.
\begin{thm}
\label{thm:adjont_periodic}We have 
\[
-\left(B^{\#}\right)^{\ast}\subseteq A^{\#}.
\]

Moreover, the following statements are equivalent:

\begin{enumerate}[(i)]

\item $-(B^{\#})^{\ast}=A^{\#},$

\item $\ker(A)+\dom(A_{0})$ is a core for $A^{\#}$,

\item $1\notin P\sigma(B^{\#}A^{\#}).$

\end{enumerate}

Moreover, the statements (i)-(iii) hold, if $\ran(A^{\#})$ is closed. 
\end{thm}

\begin{rem}
We note that statement (i) in \prettyref{thm:adjont_periodic} is
symmetric in $A$ and $B$, so we can replace the statements (ii)
and (iii) by the analogous statements involving $B$ instead of $A$.
By the same reason, the closedness of $\ran(B^{\#})$ would also be
sufficient for statement (i) to be true.
\end{rem}

\begin{example}
Let $S\subseteq\dom(S)\subseteq H\to H$ be a symmetric operator on
some Hilbert space $H$ with $S\geq cI$ for some $c>0.$ Then we
set $A_{0}\coloneqq\overline{S}$ and $B_{0}\coloneqq-\overline{S}$
and by the symmetry of $S$, we infer $A_{0}\subseteq-B_{0}^{\ast}\eqqcolon A,$
where $A=S^{\ast}.$ Hence, we are in the setting of this section.
Moreover, since $A_{0}\geq cI$ for some $c>0,$ we infer that $A_{0}$
has closed range and hence, so has $A^{\#}$ (see also \prettyref{prop:closed_range_via_A}
below). Thus, by \prettyref{prop:core_for_A=000023} 
\[
A^{\#}=\overline{S^{\ast}|_{\dom(\overline{S})+\ker(S^{\ast})}}=\overline{S^{\ast}|_{\dom(S)+\ker(S^{\ast})}}
\]
and by \prettyref{thm:adjont_periodic}
\[
A^{\#}=-(B^{\#})^{\ast}=\left(A^{\#}\right)^{\ast}.
\]
Thus, $A^{\#}$ is a selfadjoint extension of $S$ known as the Krein--von
Neumann extension of $S$ (see e.g. \cite[Section 13.3]{Schmuedgen2012}
or \cite{Fucci2022} and the references therein). 
\end{example}

We emphasise that the closedness of $\ran(A^{\#})$ is just a sufficient
condition as the following example shows.
\begin{example}
Let $\Omega=\R^{3}\setminus B[0,1]=\{x\in\R^{3}\,;\,\|x\|>1\}$ and
consider the operators $\curl_{0}$ and $\curl$ on $L_{2}(\Omega)^{3}$
given as the usual vector-analytical operator with the domains 
\begin{align*}
\dom(\curl) & \coloneqq\{u\in L_{2}(\Omega)^{3}\,;\,\curl u\in L_{2}(\Omega)^{3}\},\\
\dom(\curl_{0}) & \coloneqq\overline{C_{c}^{\infty}(\Omega)}^{\dom(\curl)},
\end{align*}
where in the first domain $\curl u$ is defined in the sense of distributions
and in the second domain the closure is taken with respect to the
graph norm of $\curl$. Then $\curl_{0}\subseteq\curl$ are both closed
densely defined operators and $\curl_{0}^{\ast}=\curl$. This situation
fits in the abstract setting considered in this section by choosing
$A_{0}\coloneqq\curl_{0},B_{0}\coloneqq-\curl_{0}$ as well as $A\coloneqq\curl,B\coloneqq-\curl.$
In contrast to the bounded domain case, where the spectrum is discrete
(see \prettyref{thm:Filonov}), we have that $\sigma(\curl^{\#})=\R$
according to \cite[p.333-334]{Picard1998_exterior}. Moreover, by
\cite{Picard1998} and \cite[Theorem 2.6]{Picard1998_exterior} $\curl^{\#}$
is indeed selfadjoint and hence, $\ran(\curl^{\#})$ cannot be closed
by \prettyref{cor:isolated_value}.
\end{example}

\begin{rem}
The construction in this example extends to the exterior derivative
$d$ on $q$-forms as an operator in $L_{2}^{q}\left(M\right)$ with
$M$ a bounded open subset of an $N$-dimensional Riemannian Lipschitz
manifold with $N=2q+1$.

With the above mechanism we get
\[
*\d^{\#}\text{ is }\begin{cases}
\text{selfadjoint for } & q\text{ odd,}\\
\text{skew-selfadjoint for } & q\text{ even},
\end{cases}
\]
where $\d$ denotes the exterior derivative on $q$-forms and $\ast$
is the Hodge-star-operator. These cases correspond to $N=4k+3$ and
$N=4k+1$ with $k=\left\lfloor \frac{q}{2}\right\rfloor $, respectively.
In particular, for $q=1$ we recover the force-free magnetic field
case. For example, if $q=0,2$, i.e. $N=1$ or $N=5$, respectively,
we have that $*\d^{\#}$ is skew-selfadjoint. The case $N=1$ recovers
the standard 1-dimensional periodic boundary case (on finite intervals).
In both these cases (and in contrast to $N=3$), we have well-posedness
of the \emph{real}, (i.e. commuting with conjugation) evolutionary
(in the sense of \cite{Picard_McGhee}) problem
\[
\left(\partial_{t}+*\d^{\#}\right)u=f.
\]
In fact, we can insert any suitable material law operator as a coefficient
of $\partial_{t}$ (see \cite{Picard2009,STW2022}). The main problem
lies in the proof of the closedness of $\ran(\d^{\#})$, which can
be shown with the help of compact embedding results for the exterior
derivative and we refer to \cite{Weck1974,PWW2001} for sufficient
conditions. 
\end{rem}

\section{Conditions for a closed range of $A^{\#}$\label{sec:Conditions-for-a}}

We recall the setting from the previous section. Let $H_{0},H_{1}$
be two Hilbert spaces and $A_{0}$ and $A$ densely defined closed
linear operators with $A_{0}\subseteq A$ and $B_{0}\coloneqq-A^{\ast}$
and $B\coloneqq-A_{0}^{\ast}.$ Moreover, we set $A^{\#}\subseteq A$
with domain 
\[
\dom(A^{\#})\coloneqq\{x\in\dom(A)\,;\,Ax\in\overline{\ran}(A_{0})\}
\]
and analogously, we define $B^{\#}.$ By \prettyref{thm:adjont_periodic}
we know that $-\left(B^{\#}\right)^{\ast}=A^{\#}$ if we can ensure
that $\ran(A^{\#})$ is closed. This section is devoted to some conditions
ensuring the closedness of $\ran(A^{\#}).$ Moreover, we will address
the question, whether $D_{A_{\red}^{\#}}$ is compactly embedded into
$H_{0}$, which by \prettyref{prop:reduced_op} would also imply the
closedness of $\ran(A^{\#})$. 

We start with a simple observation.
\begin{prop}
\label{prop:closed_range_via_A}If $\ran(A_{0})$ or $\ran(A)$ is
closed, then $\ran(A^{\#})$ is closed as well. 
\end{prop}

\begin{proof}
Note that $\ran(A_{0})\subseteq\ran(A^{\#})\subseteq\overline{\ran}(A_{0})$
by definition, and hence, if $\ran(A_{0})$ is closed, we infer that
$\ran(A^{\#})=\ran(A_{0})$ is closed. If on the other hand $\ran(A)$
is closed and we have a sequence $(x_{n})_{n\in\N}$ in $\dom(A^{\#})$
such that $A^{\#}x_{n}\to y$ for some $y\in H,$ we derive $y=Ax$
for some $x\in\dom(A)$, since $A^{\#}\subseteq A.$ Since $Ax=y=\lim_{n\to\infty}A^{\#}x_{n}\in\overline{\ran}(A_{0}),$
we obtain $x\in\dom(A^{\#})$ and $y=A^{\#}x\in\ran(A^{\#}),$ yielding
the claim. 
\end{proof}
An analogous result holds for the compact embedding of $\dom(A^{\#})\cap\ker(A^{\#})^{\bot}$
into $H_{0}.$
\begin{prop}
\label{prop:compact_via_A}We have $D_{A_{\red}^{\#}}\hookrightarrow\hookrightarrow H_{0}$
if 
\begin{itemize}
\item $D_{(A_{0})_{\red}}\hookrightarrow\hookrightarrow H_{0}$ or, 
\item $D_{A_{\red}}\hookrightarrow\hookrightarrow H_{0}.$ 
\end{itemize}
\end{prop}

\begin{proof}
Assume first that $D_{A_{\red}}\hookrightarrow\hookrightarrow H_{0}.$
Since $\ker(A^{\#})=\ker(A),$ we infer that $D_{A_{\red}^{\#}}$
is a closed subspace of $D_{A_{\red}}$ and hence, it is also compactly
embedded into $H_{0}.$\\
Assume now that $D_{\left(A_{0}\right)_{\red}}\hookrightarrow\hookrightarrow H_{0}$.
By \prettyref{cor:closed_adjoint} (b) we know that $D_{B_{\red}}\hookrightarrow\hookrightarrow H_{1}$
and hence, $D_{B_{\red}^{\#}}\hookrightarrow\hookrightarrow H_{1}$
by the first part of the proof. However, since $\ran(A_{0})$ is closed
by \prettyref{prop:reduced_op} (b), we infer that $\ran(A^{\#})$
is closed by \prettyref{prop:closed_range_via_A} and thus, $A^{\#}=(-B{}^{\#})^{\ast}$
according to \prettyref{thm:adjont_periodic}. Employing \prettyref{cor:closed_adjoint}
(b) again, we infer that $D_{A_{\red}^{\#}}\hookrightarrow\hookrightarrow H_{0}.$ 
\end{proof}
For applications it will be useful to study the case, where the operator
$A_{0}$ can be extended by an operator $\mathcal{A}$ on a larger
space, which has closed range. 
\begin{prop}
\label{prop:closed_range_via_extension}Let $\mathcal{A}:\dom\left(\mathcal{A}\right)\subseteq X\to Y$
be a densely defined closed linear operator between two Hilbert spaces
$X,Y$. Moreover, assume that $\ran(\mathcal{A})$ is closed and there
exist isometries $\iota_{X}:H_{0}\to X$, $\iota_{Y}:H_{1}\to Y$
such that 
\[
\iota_{Y}A_{0}\subseteq\mathcal{A}\iota_{X}\text{ and }\iota_{X}B_{0}\subseteq-\mathcal{A}^{\ast}\iota_{Y}.
\]
 Then, $A^{\#}$ has closed range. Moreover, if $D_{\mathcal{A}_{\red}}\hookrightarrow\hookrightarrow X,$
then $D_{A_{\red}^{\#}}\hookrightarrow\hookrightarrow H_{0}.$
\end{prop}

\begin{proof}
We need to show that $\ran(A^{\#})=\overline{\ran}(A_{0}).$ So let
$f\in\overline{\ran}(A_{0})$. Then there exists a sequence $(\phi_{n})_{n}$
in $\dom(A_{0})$ such that $A_{0}\phi_{n}\to f.$ Consequently, $\iota_{Y}A_{0}\phi_{n}\to\iota_{Y}f$
and thus, by assumption $\mathcal{A}\iota_{X}\phi_{n}\to\iota_{Y}f.$
Since $\mathcal{A}$ has closed range, we infer
\begin{equation}
\exists v\in\dom(\mathcal{A}):\:\mathcal{A}v=\iota_{Y}f.\label{eq:f_in_A}
\end{equation}
Next, we show that $\iota_{X}^{\ast}v\in\dom(A).$ For this, let $\psi\in\dom(A^{\ast})=\dom(B_{0})$
and compute 
\begin{align*}
\langle-B_{0}\psi,\iota_{X}^{\ast}v\rangle_{H_{0}} & =-\langle\iota_{X}B_{0}\psi,v\rangle_{X}\\
 & =\langle\mathcal{A}^{\ast}\iota_{Y}\psi,v\rangle_{X}\\
 & =\langle\iota_{Y}\psi,\mathcal{A}v\rangle_{Y}\\
 & =\langle\psi,\iota_{Y}^{\ast}\mathcal{A}v\rangle_{H_{1}}\\
 & =\langle\psi,f\rangle_{H_{1}}.
\end{align*}
Hence, indeed $\iota_{X}^{\ast}v\in\dom(A)$ and $A\iota_{X}^{\ast}v=f\in\overline{\ran}(A_{0}).$
Thus, we have 
\begin{equation}
\iota_{X}^{\ast}v\in\dom(A^{\#})\text{ with }A^{\#}\iota_{X}^{\ast}v=f.\label{eq:closednedd_A}
\end{equation}

Assume now that $D_{\mathcal{A}_{\red}}\hookrightarrow\hookrightarrow X$.
Then $\mathcal{A}_{\red}$ is compactly invertible by \prettyref{prop:reduced_op}.
For $f\in\ran(A^{\#})$ we have by \prettyref{eq:f_in_A} and \prettyref{eq:closednedd_A}
$A^{\#}v=f$ for $v\coloneqq\iota_{X}^{\ast}\mathcal{A}_{\red}^{-1}\iota_{Y}f.$
Hence, 
\[
\left(A_{\red}^{\#}\right)^{-1}=P\iota_{X}^{\ast}\mathcal{A}_{\red}^{-1}\iota_{Y},
\]
where $P$ denotes the orthogonal projector onto $\ker(A)^{\bot}=\ker(A^{\#})^{\bot}.$
Since $\mathcal{A}_{\red}^{-1}$ is compact, so is $(A_{\red}^{\#})^{-1}$
and thus $D_{A_{\red}^{\#}}\hookrightarrow\hookrightarrow H_{0}$
by \prettyref{prop:reduced_op}. 
\end{proof}
\begin{rem}
It is remarkable that the latter theorem just yields the closedness
of $\ran(A^{\#})$ and neither of $\ran(A_{0})$ nor of $\ran(A).$
Indeed, the arguments used in the proof suggest that the closedness
of the range can only be achieved for the operator $A^{\#}$ and not
for any other extension of $A_{0}$ by such an extension technique.
However, we were not able to construct a concrete example for an operator
$A$, such that \prettyref{prop:closed_range_via_extension} is applicable
but $\ran(A_{0})$ and $\ran(A)$ are not closed. 
\end{rem}

\section{Applications\label{sec:Applications}}

Besides the already mentioned application to the transport equation
with periodic boundary condition, we will consider Maxwell's equations
as well as the heat and wave equation with abstract periodic boundary
conditions.

\subsection{The operators $\grad^{\#}$ and $\dive^{\#}$ }

Let $\Omega\subseteq\R^{n}$ be open. As in \prettyref{rem:grad_div}
we define the operators 
\[
\grad_{c}\colon C_{c}^{\infty}(\Omega)\subseteq L_{2}(\Omega)\to L_{2}(\Omega)^{n},\quad\phi\mapsto(\partial_{j}\phi)_{j\in\{1,\ldots,n\}}
\]
and 
\[
\dive_{c}\colon C_{c}^{\infty}(\Omega)^{n}\subseteq L_{2}(\Omega)^{n}\to L_{2}(\Omega),\quad\Psi\mapsto\sum_{j=1}^{n}\partial_{j}\Psi_{j}.
\]
Clearly, both operators are densely defined and linear and by integration
by parts we infer $\grad_{c}\subseteq-\dive_{c}.$ Hence, we are in
the framework of \prettyref{sec:Abstract-Periodicity} with $A_{c}\coloneqq\grad_{c}$
and $B_{c}\coloneqq\dive_{c}.$ As in the abstract setting, we define
\[
\grad_{0}\coloneqq\overline{\grad_{c}},\:\dive_{0}\coloneqq\overline{\dive_{c}},\quad\grad\coloneqq-\dive_{0}^{\ast},\;\dive\coloneqq-\grad_{0}^{\ast}.
\]
The domains are then given by 
\[
\dom(\grad)=H^{1}(\Omega),\quad\dom(\dive)=\{\Psi\in L_{2}(\Omega)^{n}\,;\,\sum_{j=1}^{n}\partial_{j}\Psi_{j}\in L_{2}(\Omega)\}
\]
and likewise for $\grad_{0},\dive_{0},$ where the classical boundary
conditions $u=0$ on $\partial\Omega$ and $\Psi\cdot n=0$ on $\partial\Omega$
are implemented in a generalised sense (if $\partial\Omega$ is smooth
enough, one can make sense of these boundary traces, see e.g. \cite[Section 2.4]{Necas}).
Now consider the operators $\grad^{\#}$ and $\dive^{\#}$ with the
domains 
\begin{align*}
\dom(\grad^{\#}) & =\{u\in\dom(\grad)\,;\,\grad u\in\overline{\ran}(\grad_{0})\},\\
\dom(\dive^{\#}) & =\{\Psi\in\dom(\dive)\,;\,\dive\Psi\in\overline{\ran}(\dive_{0})\}.
\end{align*}

\begin{thm}
Assume that $\Omega$ is bounded in one dimension; that is, there
exists a vector $x\in\R^{n}$ with $\|x\|=1$ and a number $m>0$
such that 
\[
\Omega\subseteq\{y\in\R^{n}\,;\,|\langle y,x\rangle|\leq m\}.
\]
Then $\ran(\grad^{\#})$ is closed and $\grad^{\#}=-\left(\dive^{\#}\right)^{\ast}.$
Moreover, if $\Omega$ is bounded, then $D_{\grad_{\red}^{\#}}\hookrightarrow\hookrightarrow L_{2}(\Omega)$
and $D_{\dive_{\red}^{\#}}\hookrightarrow\hookrightarrow L_{2}(\Omega)^{n}$. 
\end{thm}

\begin{proof}
Since $\Omega$ is bounded in one dimension, we find a constant $c>0$
such that 
\[
\|u\|_{L_{2}(\Omega)}\leq c\|\grad_{0}u\|_{L_{2}(\Omega)^{n}}\quad(u\in\dom(\grad_{0})).
\]
This is Poincare's inequality (see e.g. \cite[Proposition 11.3.1]{STW2022}).
Hence, $\grad_{0}$ has a closed range and thus, $\ran(\grad^{\#})$
is closed by \prettyref{prop:closed_range_via_A}. Thus, $\grad^{\#}=-\left(\dive^{\#}\right)^{\ast}$
follows from \prettyref{thm:adjont_periodic}. \\
Moreover, if $\Omega$ is bounded, we have $D_{\grad_{0}}\hookrightarrow\hookrightarrow L_{2}(\Omega)$
by Rellich's selection theorem (see e.g. \cite[Chapter 5.7]{Evans}
or \cite[Theorem 14.2.5]{STW2022}) and thus, the compactness of $D_{\grad_{\red}^{\#}}\hookrightarrow\hookrightarrow L_{2}(\Omega)$
and $D_{\dive_{\red}^{\#}}\hookrightarrow\hookrightarrow L_{2}(\Omega)^{n}$
follows from \prettyref{cor:closed_adjoint} (b) and \prettyref{prop:compact_via_A}.
\end{proof}
It is remarkable that the closedness of $\ran(\grad^{\#})$ and the
compact embedding does not depend on the regularity of the boundary
of $\Omega,$ which is needed for analogous results for $\grad$.
The above result can be used to study diffusion or wave phenomena
with abstract periodic boundary conditions. For instance, the heat
or the wave equation 
\[
\partial_{t}u-\dive^{\#}\grad^{\#}u=f,
\]
\[
\partial_{t}^{2}u-\dive^{\#}\grad^{\#}u=f,
\]
are well-posed, due to the selfadjointness of the operator $\dive^{\#}\grad^{\#}.$
Moreover, if $\Omega$ is bounded, the solution can be computed explicitly
by using eigenvalue expansions for $\dive^{\#}\grad^{\#}.$

\subsection{The operator $\curl^{\#}$}

Let $\Omega\subseteq\R^{3}$ open. Similar to $\grad$ and $\dive$
we define the operator
\[
\curl_{c}\colon C_{c}^{\infty}(\Omega)^{3}\subseteq L_{2}(\Omega)^{3}\to L_{2}(\Omega)^{3},\,\Psi\mapsto(\partial_{2}\Psi_{3}-\partial_{3}\Psi_{2},\partial_{3}\Psi_{1}-\partial_{1}\Psi_{3},\partial_{1}\Psi_{2}-\partial_{2}\Psi_{1}).
\]
Again, this is clearly a densely defined linear operator and integration
by parts gives $\curl_{c}\subseteq\curl_{c}^{\ast}.$ Hence, we are
in the framework of \prettyref{sec:Abstract-Periodicity} with $A_{c}=-B_{c}=\curl_{c}.$
Thus, we may define 
\[
\curl_{0}\coloneqq\overline{\curl_{c}},\,\curl\coloneqq\curl_{0}^{\ast}
\]
where 
\[
\dom(\curl)\coloneqq\{\Psi\in L_{2}(\Omega)^{3}\,;\,\curl\Psi\in L_{2}(\Omega)^{3}\}
\]
and the elements in $\dom(\curl_{0})$ satisfy a generalised electric
boundary condition $n\times\Psi=0$ on $\partial\Omega$ (see \cite{Costabel}
for trace spaces associated with $\curl$). We define the operator
$\curl^{\#}$ by 
\[
\dom(\curl^{\#})\coloneqq\{\Psi\in\dom(\curl)\,;\,\curl\Psi\in\overline{\ran}(\curl_{0})\}.
\]

\begin{thm}
\label{thm:Filonov}Assume that $\Omega$ is bounded. Then $D_{\curl_{\red}^{\#}}\hookrightarrow\hookrightarrow L_{2}(\Omega)^{3}$
and in particular $\curl^{\#}$ is selfadjoint and $\sigma(\curl^{\#})$
is discrete.
\end{thm}

\begin{proof}
We want to apply \prettyref{prop:closed_range_via_extension}. For
this, let $R>0$ be such that $\Omega\subseteq B(0,R)$ and let $\mathcal{A}$
be the operator $\curl_{0}$ established on $L_{2}(B(0,R))^{3}.$
Then $D_{\mathcal{A}_{\red}}\hookrightarrow\hookrightarrow L_{2}(B(0,R))^{3}$
(see e.g. \cite{Leis1974} or \cite{Picard1984}). We denote by $\iota\colon L_{2}(\Omega)^{3}\to L_{2}(B(0,R))^{3}$
the canonical injection, that is 
\[
\left(\iota\Psi\right)(x)\coloneqq\begin{cases}
\Psi(x) & \text{ if }x\in\Omega,\\
0 & \text{ otherwise.}
\end{cases}
\]
By \prettyref{prop:closed_range_via_extension} it suffices to check
\[
\iota\curl_{0}\subseteq\mathcal{A}\iota,\quad\iota\curl_{0}\subseteq\mathcal{A}^{\ast}\iota.
\]
Since $\mathcal{A}$ is symmetric, it suffices to show the first inclusion.
So let $\Psi\in\dom(\curl_{0}).$ Then we find a sequence $(\Psi_{n})_{n\in\N}$
in $C_{c}^{\infty}(\Omega)^{3}$ with $\Psi_{n}\to\Psi$ and $\curl\Psi_{n}\to\curl_{0}\Psi$
in $L_{2}(\Omega)^{3}.$ Thus, $\iota\Psi_{n}\to\iota\Psi$ and $\curl\iota\Psi_{n}=\iota\curl\Psi_{n}\to\iota\curl_{0}\Psi$
in $L_{2}(B(0,R))^{3}.$ Since $\iota\Psi_{n}\in C_{c}^{\infty}(B(0,R))^{3}$
for each $n\in\N,$ we infer that $\iota\Psi\in\dom(\mathcal{A})$
and $\mathcal{A}\iota\Psi=\iota\curl_{0}\Psi,$ which shows the claim. 
\end{proof}
\begin{rem}
The latter theorem provides an abstract functional analytical proof
of a variant of the main result in \cite{Filonov1999}, where the
author shows the selfadjointness of $\curl^{\#}$ and the discreteness
of its spectrum by using a similar extension procedure and potential
theory. Moreover, the author provides a Weyl asymptotic of the eigenvalues.
It should be noted that in \cite{Filonov1999}, $\Omega$ is assumed
to be of finite measure and not necessarily bounded.
\end{rem}

\begin{rem}
We again emphasise that the compact embedding of $D_{\curl_{\red}^{\#}}$
does not require any regularity of the boundary of $\Omega$, while
for a corresponding result for $\curl$ such regularity assumptions
are needed (see e.g. \cite{Weck1974,PWW2001}).
\end{rem}


\begin{thebibliography}{10}

\bibitem{Costabel}
A.~Buffa, M.~Costabel, and D.~Sheen.
\newblock On traces for {$H(curl,\Omega)$} in {L}ipschitz domains.
\newblock {\em J. Math. Anal. Appl.}, 276(2):845--867, 2002.

\bibitem{Evans}
L.~C. Evans.
\newblock {\em Partial differential equations}, volume~19 of {\em Graduate
  Studies in Mathematics}.
\newblock American Mathematical Society, Providence, RI, 1998.

\bibitem{Filonov1999}
N.~D. Filonov.
\newblock The operator rot in domains of finite measure.
\newblock {\em Zap. Nauchn. Sem. S.-Peterburg. Otdel. Mat. Inst. Steklov.
  (POMI)}, 262(Issled. po Line\u{\i}n. Oper. i Teor. Funkts. 27):227--230, 236,
  1999.

\bibitem{Fucci2022}
G.~Fucci, F.~Gesztesy, K.~Kirsten, L.~L. Littlejohn, R.~Nichols, and
  J.~Stanfill.
\newblock {The Krein--von Neumann extension revisited}.
\newblock {\em Applicable Analysis}, 101(5):1593--1616, 2022.

\bibitem{Leis1974}
R.~Leis.
\newblock Rand- und {E}igenwertaufgaben in der {T}heorie elektromagnetischer
  {S}chwingungen.
\newblock {\em Z. Angew. Math. Mech.}, 54:T36--T40, 1974.

\bibitem{Necas}
J.~Ne\v{c}as.
\newblock {\em Direct methods in the theory of elliptic equations}.
\newblock Springer Monographs in Mathematics. Springer, Heidelberg, 2012.
\newblock Translated from the 1967 French original by Gerard Tronel and Alois
  Kufner, Editorial coordination and preface by \v{S}\'{a}rka Ne\v{c}asov\'{a}
  and a contribution by Christian G. Simader.

\bibitem{Pauly2019}
D.~Pauly.
\newblock A global div-curl-lemma for mixed boundary conditions in weak
  {Lipschitz} domains and a corresponding generalized
  {{\(A_0^*\)}}-{{\(A_1\)}}-lemma in {Hilbert} spaces.
\newblock {\em Analysis, M{\"u}nchen}, 39(2):33--58, 2019.

\bibitem{Picard1984}
R.~Picard.
\newblock An elementary proof for a compact imbedding result in generalized
  electromagnetic theory.
\newblock {\em Math. Z.}, 187(2):151--164, 1984.

\bibitem{Picard1998}
R.~Picard.
\newblock On a selfadjoint realization of curl and some of its applications.
\newblock {\em Ricerche Mat.}, 47(1):153--180, 1998.

\bibitem{Picard1998_exterior}
R.~Picard.
\newblock On a selfadjoint realization of curl in exterior domains.
\newblock {\em Math. Z.}, 229(2):319--338, 1998.

\bibitem{Picard2009}
R.~Picard.
\newblock A structural observation for linear material laws in classical
  mathematical physics.
\newblock {\em Math. Methods Appl. Sci.}, 32(14):1768--1803, 2009.

\bibitem{Picard_McGhee}
R.~Picard and D.~McGhee.
\newblock {\em Partial differential equations}, volume~55 of {\em De Gruyter
  Expositions in Mathematics}.
\newblock Walter de Gruyter GmbH \& Co. KG, Berlin, 2011.
\newblock A unified Hilbert space approach.

\bibitem{PWW2001}
R.~Picard, N.~Weck, and K.-J. Witsch.
\newblock Time-harmonic {M}axwell equations in the exterior of perfectly
  conducting, irregular obstacles.
\newblock {\em Analysis (Munich)}, 21(3):231--263, 2001.

\bibitem{Schmuedgen2012}
K.~Schm{\"u}dgen.
\newblock {\em Unbounded self-adjoint operators on {Hilbert} space}, volume 265
  of {\em Grad. Texts Math.}
\newblock Dordrecht: Springer, 2012.

\bibitem{STW2022}
C.~Seifert, S.~Trostorff, and M.~Waurick.
\newblock {\em Evolutionary equations}, volume 287 of {\em Operator Theory:
  Advances and Applications}.
\newblock Birkh\"{a}user/Springer, Cham, [2022] \copyright 2022.
\newblock Picard's theorem for partial differential equations, and
  applications.

\bibitem{Trostorff2014}
S.~Trostorff and M.~Waurick.
\newblock A note on elliptic type boundary value problems with maximal monotone
  relations.
\newblock {\em Math. Nachr.}, 287(13):1545--1558, 2014.

\bibitem{Weck1974}
N.~Weck.
\newblock Maxwell's boundary value problem on {R}iemannian manifolds with
  nonsmooth boundaries.
\newblock {\em J. Math. Anal. Appl.}, 46:410--437, 1974.

\end{thebibliography}
\end{document}